 \documentclass[english]{polyjmath}

\title{On the convexity of the fundamental domain of binary cyclotomic forms}

\author{\'Etienne Fouvry}
\affil{Univ. Paris--Saclay, CNRS \\
Laboratoire de Math\' ematiques d'Orsay \\
91405 Orsay, France \\
E-mail: Etienne.Fouvry@universite-paris-saclay.fr}

\author{Michel Waldschmidt}
\affil{Sorbonne Universit\' e \\
CNRS, IMJ--PRG \\
75005 Paris, France\\
E-mail: michel.waldschmidt@imj-prg.fr}

\usepackage{amsmath,amsthm}

\usepackage{amssymb}
\usepackage{graphicx}
 \usepackage[applemac]{inputenc}
\usepackage{hyperref} 
\numberwithin{equation}{section}
\def\boxit#1#2{\setbox1=\hbox{\kern#1{#2}\kern#1}%
\dimen1=\ht1 \advance\dimen1 by #1 \dimen2=\dp1 \advance\dimen2 by #1
\setbox1=\hbox{\vrule height\dimen1 depth\dimen2\box1\vrule}%
\setbox1=\vbox{\hrule\box1\hrule}%
\advance\dimen1 by .4pt \ht1=\dimen1
\advance\dimen2 by .4pt \dp1=\dimen2 \box1\relax}
\def\hfl#1#2{\smash{\mathop{\hbox to 12 mm{\rightarrowfill}}
\limits^{\scriptstyle#1}_{\scriptstyle#2}}}

 \def\hlfl#1#2{\smash{\mathop{\hbox to 12 mm{\leftarrowfill}}
\limits^{\scriptstyle#1}_{\scriptstyle#2}}}
\def\phfl#1#2{\smash{\mathop{\hbox to 8 mm{\rightarrowfill}}
\limits^{\scriptstyle#1}_{\scriptstyle#2}}}

 \def\phlfl#1#2{\smash{\mathop{\hbox to 8 mm{\leftarrowfill}}
\limits^{\scriptstyle#1}_{\scriptstyle#2}}}

\def\cqfd{\unskip\kern 6pt\penalty 500
\raise -2pt\hbox{\vrule\vbox to 10pt{\hrule width 4pt
\vfill\hrule}\vrule}\par}
\def\house#1{\setbox1=\hbox{$\,#1\,$}%
\dimen1=\ht1 \advance\dimen1 by 2pt \dimen2=\dp1 \advance\dimen2 by 2pt
\setbox1=\hbox{\vrule height\dimen1 depth\dimen2\box1\vrule}%
\setbox1=\vbox{\hrule\box1}%
\advance\dimen1 by .4pt \ht1=\dimen1
\advance\dimen2 by .4pt \dp1=\dimen2 \box1\relax}
\def\virgule{\raise 2pt \hbox{$,$}}

\def\R{\mathbb{R}}

\def\Z{\mathbb{Z}}

\def\rmd{{\mathrm d}}

\def\calO{{\mathcal{O}}}

  \newcommand{\bfm}{\mathbf {m}}

\def\og{\leavevmode\raise.3ex\hbox{$\scriptscriptstyle
\langle\!\langle$}}

\def\fg{\leavevmode\raise.3ex\hbox{$\scriptscriptstyle
\,\rangle\!\rangle$}}

\begin{document}
 
   \maketitle
  
   \hfill  %{\em \small update: \misajour}
Nous d\'edions cet article à Ren\'e Schoof, 
\\
\null   \hfill  
en témoignage de notre admiration. 

\begin{abstract}
Denote by $\Phi_m(X,Y)\in\Z[X,Y]$, $m\geqslant 1$ the sequence of binary cyclotomic forms, so that $\Phi_m(T,1)\in\Z[T]$, $m\geqslant 1$ is the sequence of cyclotomic polynomials. We prove that, for $m\geqslant 3$, the fundamental domain 
 \[
\calO_m:=
\{ (x,y)\in\R^2\; \mid \; \Phi_m(x,y)\leqslant 1\}
\]
of $\Phi_m$ is a convex subset of $\R^2$ if and only if $m$ is equal to $p$, $2p$ or $2^k$ with $p$ an odd prime and $k\geqslant 2$. Finally we investigate the convexity of the fundamental domains attached to some products of binary forms $\Phi_m$.

\end{abstract}

 \section{The result for one form}
  When $F(X,Y)\in\R[X,Y]$ is a positive definite binary form of even degree $\geqslant 2$, we denote by 
\[
\mathcal O(F):=
\{ (x,y)\in\R^2\; \mid \; F(x,y)\leqslant 1\}
\]
its fundamental domain. It is a compact subset of the Euclidean plane, stable by any homothety with ratio less or equal to $1$, in absolute value. In particular, it is symmetrical relatively to the origin
of the plane.

 The sequence of binary cyclotomic forms $\Phi_m(X,Y)\in\Z[X,Y]$, $m\geqslant 1$, is defined by induction using the formula
 \[
 X^m-Y^m=\prod_{d\mid m} \Phi_d(X,Y).
 \]
 The sequence of cyclotomic polynomials is denoted by $\phi_m(T)=\Phi_m(T,1)\in\Z[T]$, $m\geqslant 1$.
 The degree of $\Phi_m$ is $\varphi(m)$, where $\varphi$ is Euler totient function. For $m\geqslant 3$, the binary form $\Phi_m(X,Y)$ is positive definite and we denote by $\calO_m$ its fundamental domain $\mathcal O(\Phi_m)$.
 Our first goal is to prove 
 
  \begin{theorem}\label{Theorem}
For $m=p$ or $2p$ with $p$ an odd prime, and for $m=2^k$ with $k\geqslant 2$, the domain $\calO_m$ is convex.\\
 For all other values of $m\geqslant 3$, the domain $\calO_m$ is not convex. 
\end{theorem}

It follows that for a density $1$ set of $m\geqslant 3$, the domain $\calO_m$ is not convex. As $N\to\infty$, we have 
\[
\# \left\{3\leqslant m\leqslant N\; \mid \text { $\calO_m$ is convex}\right\}=\left(\frac 3 2 +o(1)\right) \frac N {\log N},
\]
which is a direct consequence of the Prime Number Theorem. 

 \section{Auxiliary results}
 
  The following values of the cyclotomic polynomials at $0$, $1$ and $-1$ are easy to check \cite[\S~1, (6) and (7)]{N},  \cite[\S~4]{FW2020}, \cite[Lemmas 2.2 and 2.3]{HM2021}. 
  
    \begin{lemma}\label{Lemma:BSMF}
We denote by $\omega(m)$ the number of distinct prime divisors of a positive integer $m$.
 We have $\phi_1(0)=-1$, $\phi_m(0)=1$ for $m\geqslant 2$, 
\[
\phi_m (1)=
\begin{cases}
0&\text{ if } m=1,\\
p& \text{ if } m=p^k\ (k\geq 1,  \text{ $p$ prime})\\
1& \text { if } \omega(m)\geqslant 2,
\end{cases}
\]
and
\[
\phi_m (-1)=
\begin{cases}
-2&\text{ if } m=1,\\ 
p& \text { if } m=2p^k\ (k\geq 1,  \text{ $p$ prime})\\
1&\text{ otherwise}.
\end{cases}
\]
 \end{lemma}
 
   Recall the general formula
\begin{equation}\label{Equation:Radical}
\phi_m (t) =\phi_q (t^{m/q}),
\end{equation}
when $q$ is the radical of $m$ (i.e. the product of all primes dividing $m$). 

 We will also use the following consequence of the Brunn--Hadamard criterion 
 \cite[Th. 2.3]{AP}, 
\cite[\S~2.3, Th. 2.10 p.~32]{G}. It links the convexity of the function $F$ to its Hessian.

 \begin{lemma}\label{Lemma:BrunnHadamard}
Let $F: \mathbb R^2 \rightarrow \mathbb R$ be a $\mathcal C^2$--function. Then $F$ is convex if and only if the Hessian matrix of $F$ is positive semi--definite. This condition is equivalent to the three inequalities
\[
\begin{cases}
F_{xx}''(x,y)\geqslant 0\\
F_{yy}''(x,y)\geqslant 0\\
F''_{xx}(x,y)F''_{yy}(x,y)-F_{xy}''(x,y)^2\geqslant 0,
\end{cases}
\]
for all $(x,y)\in \mathbb R^2$.
\end{lemma}

Let $F$ be a positive definite polynomial in $\R[X,Y]$. If $F$ is convex, then $\mathcal O (F)$ is convex. The converse is true when $F$ is homogeneous
 \cite[Theorem 5.1]{AP}. We deduce a criterion for $\mathcal O (F)$ to be convex in the particular case where $F$ is a form.

 \begin{lemma}\label{applicationofBrunnHadamard}
Let $n$ be a positive even integer and 
 $f\in\mathbb R[t]$ a monic polynomial of degree $n$ without real root. Denote by $F(X,Y)=Y^nf(X/Y)$ the associated positive definite binary form. Then 
 $\mathcal O(F)$ is convex if and only if, for all $(x,y)\in \mathcal O(F)$, we have 
\begin{equation} \label{Equation:BrunnHadamard1}
\begin{cases}
F_{xx}''(x,y)\geqslant 0\\
F_{yy}''(x,y)\geqslant 0\\
F''_{xx}(x,y)F''_{yy}(x,y)\geqslant F_{xy}''(x,y)^2.
\end{cases}
\end{equation}
\end{lemma}

\begin{remark}\label{Remark:BrunnHadamard} 
If the condition \eqref{Equation:BrunnHadamard1} is satisfied for all $(x,y)$ with $F(x,y)=1$, then by homogeneity it is satisfied for all $(x,y)\in \mathbb R^2$.
\end{remark}
\begin{remark} \label{remark2.5}Since $F$ is a form with even degree, the fundamental domain $\mathcal O (F)$ is symmetrical with respect to the origin. If furthermore the polynomial $f$ is reciprocal (as the $\phi_n$), $\mathcal O(F)$ is symmetrical with respect to the line $\{(x,y)\mid y=x\}$. This observation leads one to consider the intersection
\[
\mathcal O (F) \cap \mathcal Q,
\]
where $\mathcal Q$ is the quarter of the plane defined by
\[
{\mathcal Q } = \{ (x,y) \in \mathbb R^2 \mid \, \vert y\vert \leq x \}.
\]
\end{remark}

We now assume that the hypotheses of Lemma \ref{applicationofBrunnHadamard} are satisfied and that the polynomial $f$ is reciprocal. 

For $-1<t<1$ we set $x(t)=f(t)^{-1/n}$, so that $F(x(t),tx(t))=1$ for $-1<t<1$. The function $t\mapsto (x(t),tx(t))$ parametrizes the intersection of $\mathcal Q$ with the boundary of the fundamental domain of $\mathcal O(F)$. 

From $F(X,Y)=X^nf(Y/X)=Y^nf(X/Y)$ we deduce 
\[
F'_x(x,y)=y^{n-1}f'(x/y), \quad F''_{xx}(x,y)=y^{n-2}f''(x/y)
\]
and
\[
F'_y(x,y)=x^{n-1}f'(y/x), \quad F''_{yy}(x,y)=x^{n-2}f''(y/x).
\]
When $g$ is a polynomial of degree $d$, let us denote by $\widetilde g$ the polynomial $t^dg(1/t)$. Since $f$ is reciprocal, we have $\widetilde f=f$. In general $f'$ is not reciprocal 
 and $\widetilde{f'}\not={\tilde {f}}{\ }'$.
However ${\widetilde{f'}}'=(\rmd / \rmd t){\widetilde{f'}}(t)$ is reciprocal with degree $n-2$; indeed from 
 \[
F'_x(x,y)=y^{n-1}f'(x/y) =x^{n-1}\widetilde{f'}(y/x)
\text{ 
and
}
F'_y(x,y)=x^{n-1}f'(y/x) = y^{n-1}\widetilde{f'}(x/y)
\]
we deduce
\[
x^{n-2} {\widetilde{f'}}'(y/x)=F''_{xy}(x,y)=F''_{yx}(x,y)=y^{n-2} {\widetilde{f'}}'(x/y).
\] 
With $x=x(t)$, $y=y(t)=tx(t)$, we have
\[
F'_x(x,y)=y^{n-1}f'(1/t)=x^{n-1}\widetilde{f'}(t), \quad F''_{xx}(x,y)=y^{n-2}f''(1/t)=x^{n-2}\widetilde{f''}(t),
\]
 \[
F'_y(x,y)=x^{n-1}f'(t), \quad F''_{yy}(x,y)=x^{n-2}f''(t)
\]
and
\[
F''_{xy}(x,y)=y^{n-2}{\widetilde{f'}}'(x/y)=x^{n-2}{\widetilde{f'}}'(y/x).
\]
Therefore, thanks to Remark \ref{Remark:BrunnHadamard} and to the symmetry, the convexity conditions \eqref{Equation:BrunnHadamard1} can be written in an equivalent way as
\begin{equation} \label{Equation:BrunnHadamard2}
\begin{cases} f''(t) \geq 0\\
 \widetilde{f''}(t)\geq 0\\
f''(t) \widetilde{f''}(t) \geqslant {\widetilde{f'}}'(t)^2
\end{cases}
\text{ for } -1<t<1.
\end{equation}

Under the assumptions of Lemma \ref{applicationofBrunnHadamard},
assuming that $f$ is reciprocal and $n\geqslant 4$, let $h$, if it exists, the order of the first derivative of $f$ of order $\geqslant 2$ and $\leq n/2$, which does not vanish at $0$.
Recall that $f$ is monic and reciprocal. \\
If $h$ exists, $f(t)$ has the shape
\begin{equation}\label{f=}
 f(t)=1+a_1t+a_ht^h+\cdots 
 +a_ht^{n-h}+a_1t^{n-1}+t^n
\end{equation}
 with $2\leqslant h\leqslant n/2$ and $a_h\not=0$. \\
 If $h$ does not exist, $f(t)$ has the shape
\begin{equation}\label{f=bis}
 f(t)= 1+a_1t+a_1t^{n-1}+t^n.
\end{equation}
 We say that $\mathcal O(F)$ {\it is convex at the point $(0,1)$} if there exists $r>0$ such that the intersection of $\mathcal O(F)$ with the disc of center $(0,1)$ and radius $r$ is convex. If there exists $\varepsilon >0$ such that the three inequalities of 
\eqref{Equation:BrunnHadamard2} hold for all $t$ in the smaller interval $-\varepsilon < t <\varepsilon$, then $\mathcal O(F)$ is convex at the point $(0,1)$. \\
In the opposite direction : if for every $\varepsilon >0$, there exists $-\varepsilon < t < \varepsilon$, such that at least one the following inequalities holds $f''(t) < 0, \ 
 \widetilde{f''}(t)<0$ and 
$f''(t) \widetilde{f''}(t) < {\widetilde{f'}}'(t)^2$, then $\mathcal O(F)$ is not convex at the point $(0,1)$. 

Based on these observations, we will prove :

\begin{theorem}\label{Theorem:critereconvexite} Let $F$ be the binary form corresponding to a polynomial $f$ with even degree $n\geq 4$, without real root, monic and reciprocal, written as in \eqref{f=} or \eqref{f=bis} .
\begin{itemize}
\item Suppose that $a_1\neq 0$. Then the fundamental domain $\mathcal O(F)$ is 
\begin{enumerate} 
\item convex at the point $(0,1)$ if 
\[
a_2> \frac{n-1}{2n} a_1^2,
\]
\item not convex at the point $(0,1)$ if 
\[
a_2< \frac{n-1}{2n} a_1^2.
\]
\end{enumerate}
\item Suppose that $a_1=0$, and that $h$ exists. Then the fundamental domain $\mathcal O(F)$ is
\begin{enumerate} [label=(\roman*)]
\item convex at the point $(0,1)$ if $h$ is even and $a_h >0$, 
\item not convex at the point $(0,1)$ if $h$ is odd or if $a_h<0$.
\end{enumerate}
\item Suppose that $a_1=0$ and that $h$ does not exist. Then the fundamental domain $\mathcal O(F)$ is convex. 
\end{itemize}
\end{theorem}

 \begin{proof} 
From \eqref{f=}
we deduce the equalities 
\[ 
\begin{aligned}
f'(t)&=a_1+ha_ht^{h-1}+\cdots+ 
 (n-h)a_ht^{n-h-1}+(n-1)a_1t^{n-2}+nt^{n-1}
 \\
 f''(t)&=h(h-1)a_ht^{h-2}+\cdots+i(i-1)a_it^{i-2}+\cdots+
 \\
 &
 \qquad  (n-h)(n-h-1)a_ht^{n-h-2}+(n-1)(n-2)a_1t^{n-3}+n(n-1)t^{n-2},
 \\
 \widetilde{f'}(t)&=n+(n-1)a_1t+(n-h)a_ht^h+\cdots+ha_ht^{n-h-1}+a_1t^{n-1},
 \\
 \widetilde{f''}(t)&=n(n-1)+(n-1)(n-2)a_1t+ (n-h)(n-h-1)a_h t^h+\cdots+h(h-1)a_ht^{n-2},
 \\
 {\widetilde{f'}}'(t)&=(n-1)a_1+h(n-h)a_ht^{h-1}+\cdots+h(n-h-1)a_ht^{n-h-2}+(n-1)a_1t^{n-2}.
 \end{aligned}
\]
Notice that we have $ \widetilde{f''}(t)>0$ in the vicinity of $t=0$. We have $f''(t) \geq 0$ in the vicinity of $t=0$ if and only if $a_h >0$ and $h$ even.

\noindent $\bullet$ When $a_1\neq 0$, we have the following expansions as $t$ tends to zero
\begin{equation}\label{346}
 f''(t)\, \widetilde{f''}(t) = hn (h-1)(n-1) a_h t^{h-2} + O(t^{h-1}),
 \end{equation}
and 
\begin{equation}\label{350}
\left[ {\widetilde{f'}}'(t)\right]^2= (n-1)^2a_1^2 +O (t).
\end{equation}
The conditions \eqref{Equation:BrunnHadamard2} are satisfied for sufficiently small $t$, when $h=2$ and $a_2 > \frac{n-1}{2n} a_1^2$. 
For $h\geq 3$ (so $a_2=0$) or for $a_2 < \frac{n-1}{2n} a_1^2,$ this condition is not satisfied and $\mathcal O (F)$ is not convex at the 
point $(0,1)$.

\noindent $\bullet$ When $a_1=0$ and $h$ exists, we replace \eqref{350} by the following expansion
\[
\left[ {\widetilde{f'}}'(t)\right]^2= h^2(n-h)^2
a_h^2t^{2h-2}+O(t^{2h-1}).
\]
Comparing with \eqref{346} we see that \eqref{Equation:BrunnHadamard2} is satisfied for any sufficiently small $t$, if and only if 
$h$ is even and $a_h >0$. Recall that $a_h \neq 0$ by definition.

\noindent $\bullet$ When $a_1=0$ and $h$ does not exist, the form $F$ has the shape $F(X,Y) =X^n+Y^n$ with even $n$. It is well known that $\mathcal O(F)$ is convex. A proof of this fact is recovered from Lemma \ref{applicationofBrunnHadamard} since $n$ is even and since $F''_{xy}(x,y)=0$ while $F''_{xx}(x,y)$ and $F''_{yy}(x,y)$ are $\geqslant 0$. 

Another proof of the convexity of the fundamental domain of $F(X,Y)=X^n+Y^n$ for even $n\geqslant 2$ is as follows.
The boundary of $\mathcal O (F)$ in the quadrant $\{(x,y) \mid x\geq 0, \, y\geq 0\}$ is given by the function
$x\in [0, 1] \mapsto (x, y(x))$ with
 \[
 y(x)=(1-x^n)^{1/n}.
\]
A direct computation gives the equality
\[
y''=- (n-1) x^{n-2} (1-x^n)^{(1/n)-2},
\]
which is negative for $0< x <1$. 
Hence $\mathcal O(F)$ is convex.
 \end{proof}
\begin{remark} Theorem \ref{Theorem:critereconvexite} gives no information when $a_1\neq 0$ and $a_2= (n-1)a_1^2/(2n)$. A precise study of this case requires hypotheses on the coefficients $a_{t}$
with $t \geq h+1$.
\end{remark}

 \section{Proof of convexity in Theorem \ref{Theorem} } 
 As mentioned in Remark \ref{remark2.5} it is sufficient to check the inequality \eqref{Equation:BrunnHadamard1} for $(x,y)\in \calO_m \cap \mathcal Q$, to obtain the convexity of $\calO_m$.
 
   We begin by the case $n=2^a$ with $a\geqslant 2$, we have
\[
 \Phi_{2^a}(X,Y)=X^{2^{a-1}}+Y^{2^{a-1}}.
 \]
The convexity of $\mathcal O_{2^a}$, and more generally of the fundamental domain of the binary form $F(X,Y)=X^n+Y^n$ when $n$ is even, is a consequence of Theorem \ref{Theorem:critereconvexite} (when $a_1=1$ and $h$ does not exist).

 When $m=2 m'$ with $m'$ odd we have 
 \begin{equation}\label{Equation:2m}
 \Phi_{2m'}(X,Y)=\Phi_{m'}(X,-Y).
 \end{equation}
 Therefore the convexity part of Theorem \ref{Theorem} follows from the next result with $m=p$. 
 
  \begin{lemma}
 Let $m\geqslant 2$ be an odd integer. Then the fundamental domain of the binary form
 \[
 F(X,Y)=X^{m-1}+X^{m-2}Y+\cdots+XY^{m-2}+Y^{m-1}=\frac{X^m-Y^m}{X-Y},
 \]
 namely
\[
\mathcal O(F)=
\{ (x,y)\in\R^2\; \mid \; |x^m-y^m|\leqslant |x-y|\},
\]
 is convex.
 \end{lemma}
 
  \begin{proof}
 The idea behind the following proof was suggested by AI. 
 Let $(x,y)\in\R^2$ satisfy $|F(x,y)|\leqslant 1$. 
For $0\leqslant s\leqslant 1$ define the linear function $L(s)=sx+(1-s)y$. The derivative of $L(s)^m$ is $m(x-y)L(s)^{m-1}$.
Hence
\[
F(x,y)=m\int_0^1 L(s)^{m-1}\rmd s.
\]
We deduce 
\[
\begin{aligned}
F''_{xx}(x,y)&=m(m-1)(m-2)\int_0^1 s^2L(s)^{m-3}\rmd s,
\\
F''_{xy}(x,y)&=m(m-1)(m-2)\int_0^1 s(1-s)L(s)^{m-3}\rmd s,
\\
F''_{yy}(x,y)&=m(m-1)(m-2)\int_0^1 (1-s)^2L(s)^{m-3}\rmd s.
\end{aligned}
\]
As a consequence, for $t\in \R$, we have 
\[
F''_{xx}(x,y)t^2+2F''_{xy}(x,y)t+F''_{yy}(x,y)=
m(m-1)(m-2)\int_0^1 (st+1-s)^2L(s)^{m-3}\rmd s,
\]
which is $\geqslant 0$. Therefore the above quadratic form in $t$ is positive definite. We deduce that equation \eqref{Equation:BrunnHadamard1} holds and Lemma \ref{applicationofBrunnHadamard} can be applied.
\end{proof}
 
 \section{Proof of the non--convexity in Theorem \ref{Theorem}}
 Let $m\geqslant 3$ be an integer not of the form $m=p$, $2p$ nor $m=2^k$ with $p$ an odd prime and $k\geqslant 2$. Then $m$ satisfies one of the following conditions:
\begin{enumerate}
\item
$m$ has at least two different odd prime divisors, 
\item
$m=p^\ell$ with $p$ an odd prime and $\ell\geqslant 2$, 
\item
$m=2^a p^\ell$ with $p$ an odd prime, $a\geqslant 1$, $\ell \geqslant 1$ and $(a,\ell)\not=(1,1)$.
\end{enumerate}
In case (1) one deduces from Lemma \ref{Lemma:BSMF} 
\[
\Phi_m(0,1)=\Phi_m(1,1)=\Phi_m(-1,1)=1,
\]
hence the line $y=1$ has three points of intersection with the boundary of $\calO_m$, and therefore $\calO_m$ is not convex. 

The same is true in case (3) when $a\geqslant 2$. 

The case (3) with $a=1$ reduces to case (2), as shown by equation \eqref{Equation:2m}. 

It remains to show that in case (2) the set $\calO_m$ is not convex. Let $m=p^\ell$ with $p$ an odd prime and $\ell\geqslant 2$. 
From formula \eqref{Equation:Radical} we deduce
\[
\phi_m(t)=
\phi_{p^\ell}(t)=\phi_p(t^{p^{\ell-1}})=1+t^{p^{\ell-1}}+t^{2p^{\ell-1}}+\cdots+t^{(p-1)p^{\ell-1}},
\]
and the first non--zero derivative or order $\geqslant 2$ at $0$ of $\phi_m$ has an odd order, precisely $p^{\ell-1}$. Therefore, near $t=0$, we have $\phi''_m(t)>0$ for $t>0$, $\phi''_m(t)<0$ for $t<0$, hence the curve $\Phi_m(x,y)=1$ has an inflexion point at $(0,1)$, and therefore $\calO_m$ is not convex at this point. 

We may also remark that the convexity condition \eqref{Equation:BrunnHadamard2} is not satisfied since near $t=0$ we have $\phi''_m(t)<0$ for $t<0$, while $\widetilde{\phi''}_m(0)=\varphi(m)(\varphi(m)-1)>0$.

This completes the proof of Theorem \ref{Theorem}. 

\begin{remark} Using formula \eqref{Equation:Radical} we obtain 
for $p$ an odd prime, $a\geqslant 2$ and $\ell\geqslant 1$ the equalities 
\[
\phi_{2^a p^\ell}(t)=\phi_{2p}(t^{2^{a-1}p^{\ell-1}})= \phi_p (-t^{2^{a-1}p^{\ell -1}}) =1-t^{2^{a-1}p^{\ell-1}}+O(t^{4p^{\ell-1}}),
\] 
as $t$ tends to zero.
Theorem \ref{Theorem:critereconvexite} (with $a_1=0$, $h$ even and $a_h$ negative) shows that $\mathcal O_{2^ap^\ell}$ is not convex at $(0,1)$. 
 This gives another proof of the non--convexity of $\mathcal O_{2^ap^\ell}$ for $a\geqslant 2$. We deduce that if $\calO_m$ is not convex but is convex at $(0,1)$, then $m$ has at least two different odd prime divisors. 

 Also for $p_1$ and $p_2$ two odd primes, 
we have near $t=0$ 
\[
\phi_{p_1p_2}(t)=\frac {\phi_{p_1}(t^{p_2})}{\phi_{p_1}(t)}=1-t+O(t^3),
\] 
hence Theorem \ref{Theorem:critereconvexite} (with $a_1=-1$, $a_2=0$) shows that $\mathcal O_{p_1p_2}$ is not convex at $(0,1)$. 

On the other hand, since $105$ is the product of three odd primes, $\mathcal O_{105}$ is not convex by Theorem \ref{Theorem}. However it is convex at $(0,1)$, since we have
\[
\begin{aligned}
&\Phi_{105}(t) = 1+t+ t^2-t^5-t^6-2t^7-t^8-t^9+t^{12}+t^{13}+t^{14}+t^{15}+t^{16}+t^{17}-t^{20}-t^{22}-t^{24}
\\ 
& \quad - t^{26} - t^{28} + t^{31} +t^{32} + t^{33} + t^{34} + t^{35} + t^{36} - t^{39} - t^{40} - 2 t^{41} - t^{42} - t^{43} + t^{46} + t^{47} + t^{48} .
\end{aligned} 
\]
which allows to apply Theorem \ref{Theorem:critereconvexite} below with $a_1=1$, $a_h=a_2=1$ and $m =48$. More generally, Theorem \ref{productofPhi} with $s_1=1$, $s_2=s_3=0$ shows that when $m$ is odd with $\mu(m)=-1$, then $\calO_m$ is convex at $(0,1)$. The same is true when $m$ is even with $\mu(m)=1$.
\end{remark}

 \section{Multicyclotomy}
 A natural question is to control the convexity of the fundamental domain $\mathcal O (F_1\cdot F_2)$\break when $F_1$ and $F_2$ are two definite not proportional binary forms. This program is very ambitious~: the convexity (or not) of $\mathcal O( F_1)$ and of $\mathcal O(F_2)$ has no obvious link with the convexity (or not) of $\mathcal O (F_1\cdot F_2)$. For instance $\mathcal O (100 x^2+y^2)$ and $ \mathcal O ( x^2+ 100 y^2)$ are both convex;
 however $\mathcal O ( (100x^2+y^2) (x^2+ 100 y^2))$  is not convex.

 In \cite{FW2026}, we defined the set of {\em multicyclotomic} binary forms
 \[
 \Phi_{\bfm}(X,Y)= \Phi_{m_1}(X,Y)\Phi_{m_2}(X,Y)\dots \Phi_{m_s}(X,Y),
 \]
where $\bfm=(m_1,\dots,m_s)$ with $s\geqslant 1$, $1\leqslant m_1<\cdots<m_s$.
Such a form is totally positive if and only if $m_1\geqslant 3$. 

\bigskip
\noindent
{\bf Question.} 
{\em For which compositions $\bfm$ with $s\geqslant 1$ and $m_1\geqslant 3$ is the set
\[ 
\mathcal O_{\bfm}:=
\{ (x,y)\in\R^2\; \mid \; |\Phi_{\bfm}(x,y)|\leqslant 1\}
\] 
convex? 
}\\
Theorem \ref{Theorem} completely answers to this question when $s=1$.

\subsection{A first example of non convexity}

Let $\bfm=(m_1,\dots,m_s)$ with $s\geqslant 1$, $3\leqslant m_1<\cdots<m_s$.
Assume that for $1\leqslant i\leqslant s$, we have 

$\bullet$
either $\omega(m_i)\geqslant 2$ and $m_i\not \equiv 2 \pmod 4$

$\bullet$
or $2\Vert m_i$ and $\omega(m_i/2)\geqslant 2$.
\\
Then from Lemma \ref{Lemma:BSMF} we deduce
\[
\Phi_{\bfm}(0,1)=\Phi_{\bfm}(1,1)=\Phi_{\bfm}(-1,1)=1,
\]
hence $\mathcal O_{\bfm}$ is not convex. 

\subsection{Bicyclotomy with powers of $2$}\label{S:Bicyclotomy}
We consider the case $k=2$, $\bf m = (m_1,m_2)$ with $m_1$ and $m_2$ powers of $2$. This case of bicyclotomy seems to be hard to be treated in full generality. So we are only concerned with the convexity of the fundamental domain $\mathcal O_{2^a, 2^b}$ with $b>a$. 

We have

\begin{theorem} \label{658}Let $a>b\geq 2$ be two powers of $2$ and let $F_{a,b}(X,Y)$ be the positive binary symmetrical form
\[
F_{a,b} (X,Y) = (X^a+Y^a)(X^b+Y^b)=\Phi_{2a} (X,Y) \Phi_{2b} (X,Y).
\]
Then the corresponding fundamental domain $\mathcal O(F_{a,b})$ is convex.
\end{theorem}

\begin{proof} To shorten notations, write $F=F_{a,b}$. We have the list of equalities
\[
F= X^{a+b} +X^aY^b +X^b Y^a+ Y^{a+b},
\]
\[
F'_x = (a+b) X^{a+b-1} + a X^{a-1} Y^b + b X^{b-1} Y^a,
\]
\[
F'_y= bX^a Y^{b-1}+aX^bY^{a-1} +(a+b)Y^{a+b-1},
\]
\[
F''_{xx}= (a+b)(a+b-1) X^{a+b-2} +a (a-1) X^{a-2} Y^b +b(b-1)X^{b-2} Y^a,
\]
\[
F''_{yy}= b(b-1) X^aY^{b-2} + a (a-1) X^b Y^{a-2}+(a+b)(a+b-1) Y^{a+b-2}.
\]
\[
F''_{x,y}= ab X^{a-1}Y^{b-1}+ab X^{b-1}Y^{a-1}.
\]
The formula for the Hessian $H(x,y)$ of $H$ is
\begin{align*}H(X,Y)=F''_{xx}F''_{yy}-(F''_{xy})^2 =& b (b-1) (a+b)(a+b-1) X^{2a+b-2} Y^{b-2}\\
&-ab (a+b-1)X^{2a-2} Y^{2b-2}\\
&+a(a-1)(a+b)(a+b-1)X^{a+2b-2} Y^{a-2}\\ 
&+2 (a+b-1)((a^2+b^2)(a+b-1)-ab)X^{a+b-2} Y^{a+b-2}\\
&+a(a-1)(a+b)(a+b-1) X^{a-2} Y^{a+2b-2}\\
&-ab (a+b-1)X^{2b-2} Y^{2a-2}\\
&+b(b-1)(a+b)(a+b-1)X^{b-2}Y^{2a+b-2}\\
&= A-B+C+D+E-G+J,
\end{align*}
by definition.
We must prove that for all $(x,y)\in \mathbb R^2$ one has 
\begin{equation}\label{H>0}H(x,y)\geq 0.
\end{equation} By symmetry, we may restrict to the case $x\geq y$.
Firstly we study sufficient conditions (on $a$ and $b$) to ensure the inequaliy
\begin{equation}\label{A>B+G}A(x,y)\geq B(x,y)+G(x,y),
\end{equation} for every $x$ and $y$. 
Indeed we have, for every $x$ and $y$ and every powers of $2$, $a>b\geq 2$, the inequalities 
\begin{equation}\label{in1}x^{2a+b-2}y^{b-2}\geq x^{2a-2} y^{2b-2},
\end{equation}
and 
\begin{equation}\label{in2} x^{2a+b-2}y^{b-2} \geqslant x^{2b-2} y^{2a-2}.
\end{equation}
We now study the coefficient $A(1,1)-(B(1,1) +G(1,1))$. It has the value 
\[
\begin{aligned}
 b (b-1) (a+b)(a+b-1)-\left( ab (a+b-1)+ ab (a+b-1)\right)\nonumber \\= b(a+b-1) \left( (b-1)(a+b)-2a\right)\geq 0,\label{703}
 \end{aligned}
 \]
if we suppose that $b\geq 4$. Combining \eqref{in1}, \eqref{in2} and \eqref{703}, we proved \eqref{A>B+G}. By positivity, we proved 
\eqref{H>0} under the restriction $b\geq 4$.

We now suppose that $b=2$. Notice that we have $a\geq 4$. By the above calculations, we already know that
\begin{equation}\label{A>B} A(x,y) \geq B(x,y).
\end{equation}
 We take into account the term $E(x,y)$ in order to prove that 
 \begin{equation}\label{E>G}E(x,y) \geq G(x,y).
 \end{equation}
 We notice that we have 
\begin{equation}\label{717}
x^{a-2}y^{a+2b-2}\geq x^{2b-2}y^{2a-2}
\end{equation}
for $x\geq y $, $a \geq 4$ and $b=2$. Secondly, under the same conditions, we also have the inequality
\begin{equation}\label{721}
a(a-1)(a+b)(a+b-1)\geq ab (a+b-1)\geq 0.
\end{equation}
Combining \eqref{717} and \eqref{721}, we deduce \eqref{E>G}. By positivity and by \eqref{A>B} we deduce that
\[
H(x,y)\geq (A(x,y) -B(x,y)) +(E(x,y)-G(x,y)) \geq 0.
\]
 The proof of \eqref{H>0} is complete. This ends the proof of Theorem \ref{658}.
\end{proof} 
 
  \subsection{Convexity at the point $(0,1)$}\label{S:Near(0,1)}
  
  We recall some values of the first coefficients of the cyclotomic polynomials $\phi_m$
\cite[(4.3)]{HM2021}. 
\begin{lemma}\label{valuesofcoeff} Let $m\geq 2$ and $\phi_m(t)$ be the cyclotomic polynomial, written as 
\[
\phi_m(t) = 1+\nu_1 t+\nu_2 t^2 +O(t^3)
\]
in the vicinity of $t=0$. We then have 
\[
\nu_1 =-\mu (m),
\]
and 
\[
\nu_2 = \begin{cases}
[\mu^2 (m)-\mu (m)]/2 \text{ if } $m$ \text{ is odd,}\\
\\
[\mu^2 (m)-\mu (m)]/2-\mu (m/2) \text{ if } $m$ \text{ is even.}
\end{cases}
\]
\end{lemma}

In other terms, for $m$ odd $\geqslant 3$ we have
\[
\phi_m(t)=
\begin{cases}
1+t+t^2+O(t^3) & \text { if $\mu(m)=-1$,}
\\
1-t+ O(t^3) & \text { if $\mu(m)=1$,}
\\
1+ O(t^3) & \text { if $\mu(m)=0$}
\end{cases}
\]
while for $m$ even $\geqslant 4$ we have
\[
\phi_m(t)=
\begin{cases}
1+t+O(t^3) & \text { if $\mu(m)=-1$,}
\\
1-t+t^2+ O(t^3) & \text { if $\mu(m)=1$,}
\\
1- \mu(m/2)t^2+O(t^3) & \text { if $\mu(m)=0$.}
\end{cases}
\]
Notice that if $m$ is odd $\geqslant 3$ and $\mu(2m)\not=0$, then $\phi_{2m}(t)=\phi_m(-t)$ and $\mu(2m)=-\mu(m)$.

Combining Theorem \ref{Theorem:critereconvexite} with Lemma \ref{valuesofcoeff} we obtain the following:

\begin{theorem}\label{productofPhi} 
\begin{enumerate}
\item 
Let $s_1,s_2,s_3$ be non negative integers, $s=s_1+s_2+s_3$ and $m_1,\dots,m_s$ distinct odd integers $\geqslant 3$ satisfying
\[
\begin{cases}
&\#\{i\;\mid \; 1\leqslant i\leqslant s, \; \mu(m_i)=-1\}=s_1,
\\
&\#\{i\;\mid \; 1\leqslant i\leqslant s, \; \mu(m_i)=0\}=s_2,
\\
&\#\{i\;\mid \; 1\leqslant i\leqslant s, \; \mu(m_i)=1\}=s_3.
\end{cases}
\]
Assume $s_1>s_2$. 
Then the fundamental domain $\mathcal O (\Phi_{m_1} \cdots \Phi_{m_s})$ is convex at the point $(0,1)$.
\item Let $p_1,\dots,p_s$ be distinct odd primes and $k_1,\dots,k_s$ integers, all of them $\geqslant 2$. Set $q_i=p_i^{k_i}$ ($1\leqslant i\leqslant s$). Then the fundamental domain $\mathcal O (\Phi_{q_1} \cdots \Phi_{q_s})$ is not convex at the point $(0,1)$, hence is not convex.
\end{enumerate}
\end{theorem}

For instance from the first item with $s_2=0$ we deduce that if $p_1<p_2<\cdots<p_s$ are distinct odd primes and $k_1,\dots,k_s$ are positive integers with one at least of them equal to $1$, if we write $q_i=p_i^{k_i}$ ($1\leqslant i\leqslant s$), then the fundamental domain $\mathcal O (\Phi_{q_1} \cdots\Phi_{q_s})$ is convex at the point $(0,1)$. Item 2 shows that the assumption $\min\{k_1,\dots,k_s\}\geqslant 2$ is necessary. 

\begin{proof} 
\null\hfill
\\
1. Let $f:=\Phi_{m_1} \cdots \Phi_{m_s}$. According to Lemma \ref{valuesofcoeff} we have
\[
\begin{aligned}
f(t) &= (1+t+t^2)^{s_1}(1-t)^{s_2} +O(t^{3})
\\
&= 1+(s_1-s_2)t+ \frac 1 2 ({s_1(s_1+1)} + s_2(s_1-1)-2s_1s_2) t^2+O(t^{3})
\\
&= 1+(s_1-s_2)t+ \frac 1 2 (s_1-s_2)(s_1-s_2+1) t^2+O(t^{3}).
\end{aligned}
\]
Since $s_1>s_2$, we have
\[
(s_1-s_2)(s_1-s_2+1)/2>(s_1-s_2)^2(n-1)/2n.
\]
Item 1 then follows from Theorem \ref{Theorem:critereconvexite}. 
\\
2. From formula \eqref{Equation:Radical} we deduce that for $p$ an odd prime and $k\geqslant 2$ we have
\[
\Phi_{p^k} (t) = 1+ t^{p^{k-1}}+O(t^{2p^{k-1}}).
\]
Assume without loss of generality $p_1^{k_1-1}<\cdots<p_s^{k_s-1}$. We deduce
\[
\Phi_{q_1}(t) \cdots\Phi_{q_s}(t)= 1 + t^{p_1^{{k_1-1}}}+O(t^{p_1^{k_1-1}+1})
\]
and we apply Theorem \ref{Theorem:critereconvexite} with $a_1=0$, and the odd value of $h=p_1^{k_1-1}$.
\end{proof}

Further examples could easily be given, but a general picture is not clear. 
%%%%%%%%%%%%%%%%%%%%%
%%%%%%%%%%%%%%%%%%%%%%%%%%%%%%%%%%%%%%%%%%%%%%%%%%%%%%%%%%%%%% 

\section*{Pictures}

{\footnotesize
Here are some pictures of the fundamental domain $\calO_m$: for $m\in\{3,7\}$, $m\in\{6,10\}$ and $m\in\{4,8\}$ the domain $\calO_m$ is convex, for $m\in\{15, 105\}$, $m\in\{9, 25\}$ and $m\in\{12, 20\}$ the domain $\calO_m$ is not convex. 

Recall Theorem 1.5 of \cite{FW2020}: when $m$ is large, $m>m_0(\epsilon)$, the fundamental domain $\calO_m$ contains the square centered at the origin and side $2-m^{-1+\epsilon}$ and is contained in the square centered at the origin and side $2+m^{-1+\epsilon}$. 
 In particular, for a very large $m$, it is difficult to guess only by looking at the picture whether $\calO_m$ is convex or not.} 
 
  \includegraphics[height= 15.5 true cm]{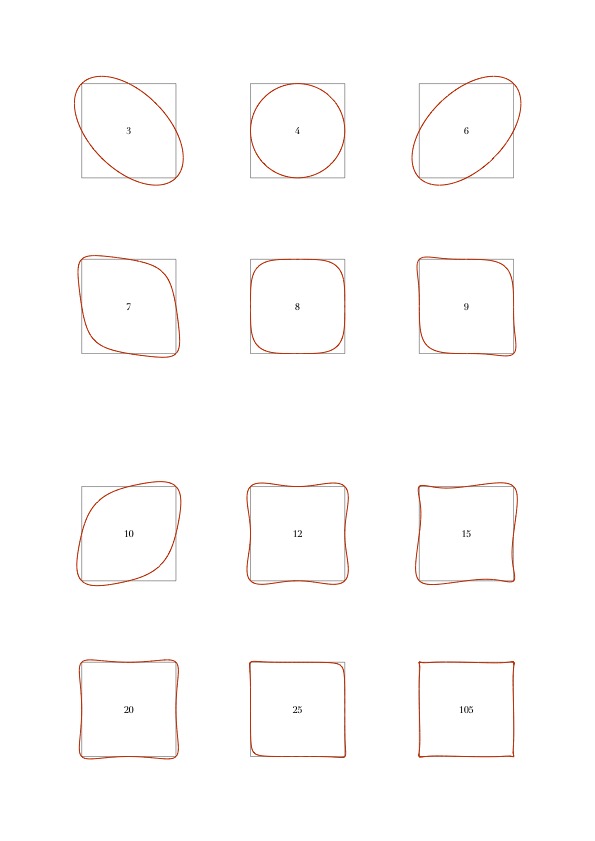}

\newpage

A picture of the fundamental domain $\mathcal O_{2^3,2^4}$, which is convex (\S \ref{S:Bicyclotomy}). 
\bigskip
\begin{center}

\includegraphics[height= 4 true cm]{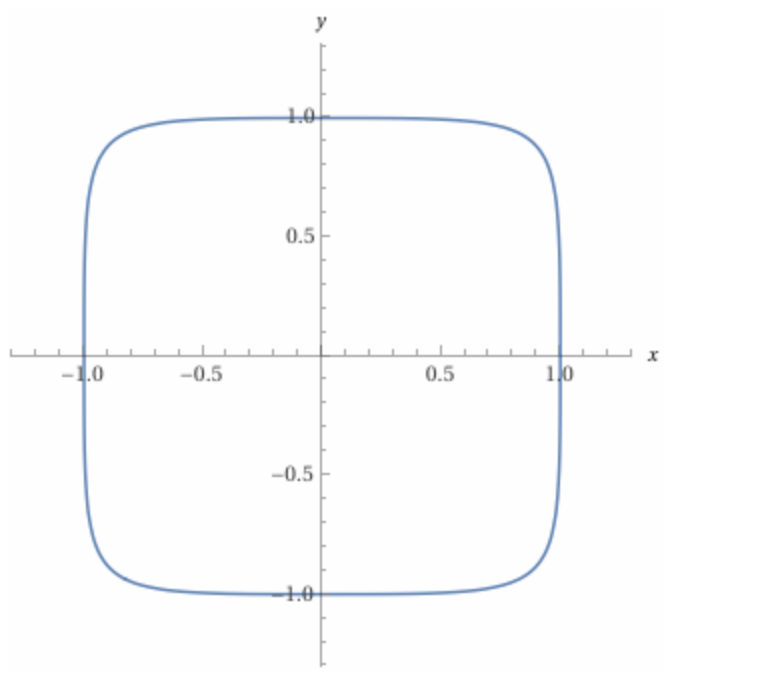}

\end{center}

\bigskip

Here is a picture of the fundamental domain $\mathcal O_{3,5}$, which is convex near $(0,1)$ (\S \ref{S:Near(0,1)}). 
\bigskip
\begin{center}

\includegraphics[height= 4 true cm]{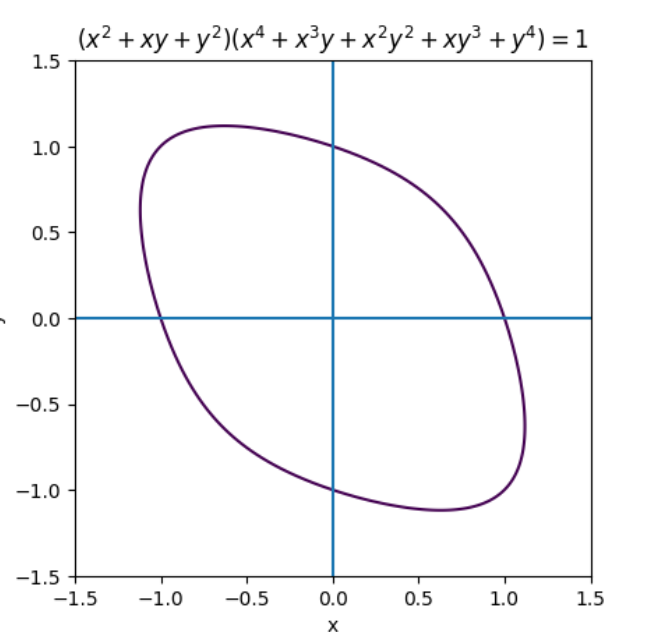}

\end{center}

 \vfill\eject
  \end{document}